\documentclass[11pt]{amsart}

\usepackage[T1]{fontenc}
\usepackage{lmodern}
\usepackage{amsmath,amssymb,amsthm,mathtools}
\usepackage{microtype}
\usepackage[numbers,sort&compress]{natbib}
\usepackage{enumitem}
\usepackage[colorlinks=true,citecolor=blue,linkcolor=blue,urlcolor=blue]{hyperref}

\numberwithin{equation}{section}
\newtheorem{theorem}{Theorem}[section]
\newtheorem{proposition}[theorem]{Proposition}
\newtheorem{lemma}[theorem]{Lemma}
\newtheorem{corollary}[theorem]{Corollary}
\theoremstyle{definition}
\newtheorem{definition}[theorem]{Definition}
\newtheorem{assumption}[theorem]{Assumption}
\newtheorem{remark}[theorem]{Remark}
\newtheorem{example}[theorem]{Example}

\newcommand{\R}{\mathbb R}

\newcommand{\M}{\mathcal M}
\newcommand{\Kclass}{\mathfrak K}
\newcommand{\calA}{\mathcal A}
\newcommand{\calD}{\mathcal D}
\newcommand{\calK}{\mathcal K}
\newcommand{\norm}[1]{\left\lVert #1\right\rVert}

\newcommand{\dual}[2]{\left\langle #1,#2\right\rangle}
\newcommand{\weakstar}{\stackrel{*}{\rightharpoonup}}
\DeclareMathOperator{\dist}{dist}
\DeclareMathOperator{\Lip}{Lip}
\newcommand{\dBL}{d_{\mathrm{BL}}}
\providecommand{\doi}[1]{\href{https://doi.org/#1}{doi:\nolinkurl{#1}}}
\title[Kernel-robust dynamics with measure-valued delay]{Kernel-Robust Dynamics for Reaction--Diffusion Equations with Measure-Valued Delay}
\author{Lennon J. Shikhman}
\address{Department of Mathematics and Systems Engineering, Florida Institute of Technology, Melbourne, FL, USA}
\email{lshikhman2022@fit.edu}
\date{}

\begin{document}
\begin{abstract}
We study reaction--diffusion equations with finite signed measure delays and nonlinear, possibly nonlocal feedback. We treat abstract reactions that are Lipschitz on bounded $L^2$-balls with quadratic coercivity, and pointwise cubic polynomials with positive leading coefficient in spatial dimensions at most three. For affine-growth feedback and time-independent forcing in $H^{-1}(\Omega)$, both classes generate global weak solution semiflows on $X=C([-r,0];L^2(\Omega))$. Besides total-variation Lipschitz stability, we obtain quantitative weak-star stability with modulus $\varpi(d)=d\log(e+d^{-1})$ for the abstract class and $\varpi(d)^{1/2}$ for cubic reactions, where $d$ is the bounded-Lipschitz distance between delay measures. The estimates are uniform on bounded sets of continuous histories and yield concentration and atomic-quadrature rates. A damping condition gives common compact absorption for the abstract class. Cubic coercivity removes that smallness condition on every fixed total-variation-bounded kernel class. In both cases the global attractors are upper semicontinuous in $X$ and in $C([-r,0];H_0^1(\Omega))$.
\end{abstract}
\maketitle

\section{Introduction}

Reaction--diffusion equations with delay describe systems in which present evolution depends on a finite history of past states. A memory law may be estimated from data, regularized, replaced by a discrete delay, or approximated by quadrature. These operations raise two questions: how accurately do the resulting solutions approximate the original trajectories, and how do the long-time dynamics change? Partial functional differential equations provide a natural framework for these questions; see Travis and Webb \cite{traviswebb1974}, Martin and Smith \cite{martinsmith1991}, and the monographs \cite{haleverduyn1993,wu1996}.

We consider
\begin{equation}\label{eq:intro}
 \partial_tu-\Delta u+f(u)
 =\int_{[-r,0]}B(u(t+\theta))\,d\mu(\theta)+h,
\end{equation}
with homogeneous Dirichlet boundary conditions on a bounded Lipschitz domain $\Omega\subset\R^d$ and a prescribed history in $C([-r,0];L^2(\Omega))$. Here $\mu$ is a finite signed Borel measure. Absolutely continuous measures give distributed delays, point masses give discrete delays, and an atom at zero gives instantaneous feedback. Signed measures allow combinations of positive and negative feedback without imposing order-preserving structure.

Measure-valued delay is an established formulation. It appears in infinite-dimensional control theory \cite{bensoussandapratodelfourmitter2007}. Casas, Mateos and Tr\"oltzsch \cite{casasmateostroltzsch2018} prove well-posedness, differentiability of the measure-to-state map, and strong state convergence under weak-star kernel convergence for a semilinear parabolic control problem with linear delayed feedback. Their setting includes polynomial scalar reactions and spatially continuous histories. Our results permit nonlinear, possibly nonlocal feedback on $L^2(\Omega)$ and histories that are only continuous with values in that space. Section~\ref{sec:cubic} treats pointwise cubic reactions in this setting. The quantitative estimates are uniform over bounded sets of such histories.

Ishizaka \cite{ishizaka2026wellposedness} treats coercive evolution equations with form-valued delayed feedback that can contain principal spatial derivatives. The revised framework gives signed-kernel total-variation stability and narrow stability for nonnegative retarded kernels at natural energy regularity. That setting and the present one are complementary: here the feedback is $L^2$-valued, and weak-star stability allows signed kernels and atoms at zero. Continuous dependence under weak coefficient topologies has also been developed for linear nonautonomous parabolic delay equations by Kryspin and Mierczy\'nski \cite{kryspinmierczynski2023,kryspinmierczynski2024}.

The long-time theory of delayed equations is extensive; examples include \cite{caraballoreal2004,caraballomarinvalero2005,marinrubioreal2010}. Of particular relevance, Pra\v z\'ak \cite{prazak2007} constructs exponential attractors for abstract parabolic systems whose history dependence is controlled by a possibly singular measure, assuming a bounded invariant region. Thus, neither singular measure delays nor attractor existence alone distinguish the present work. Our focus is quantitative perturbation of the delay law, explicit uniform absorption, and convergence of the associated attractors under that perturbation. The abstract attractor machinery is classical \cite{hale1988,temam1988,robinson2001,babinvishik1992,chepyzhovvishik2002,carvalholangarobinson2013}, as is the upper-semicontinuity mechanism \cite{halelinraugel1988}.

Finite delay should also be distinguished from heat-conduction and viscoelastic memory models, in which derivatives of the state often enter a convolution on an infinite interval. Those models naturally lead to history or minimal-state spaces; see \cite{dafermos1970,giorgipatamarzocchi1998,grassellipata2006,contimarchinipata2014,pruss1993}. We retain the continuous finite-history space throughout.

The main results are the following.
\begin{enumerate}[label=(\roman*)]
 \item For the abstract reaction class, global weak well-posedness holds under bounded-ball Lipschitz assumptions, affine growth only for the feedback, and a quadratic reaction coercivity bound. No global one-sided Lipschitz condition or upper growth bound on the reaction is needed.
 \item For that class, the solution semiflow is Lipschitz in total variation and has modulus $d\log(e+d^{-1})$ in the bounded-Lipschitz distance. The latter estimate is uniform on bounded sets of histories and gives concentration and atomic-quadrature rates.
 \item A damping condition involving the feedback growth slope yields common compact absorption and compact global attractors with forcing $h\in H^{-1}(\Omega)$. The threshold is sharp as a guarantee based only on the abstract class's structural constants.
 \item The attractors are upper semicontinuous in both $C([-r,0];L^2(\Omega))$ and $C([-r,0];H_0^1(\Omega))$. A structured negative-feedback family admits a stronger criterion that allows every fixed feedback gain.
 \item A separate variational argument covers every scalar cubic polynomial with positive leading coefficient when $d\leq3$, including the Allen--Cahn reaction. It gives global well-posedness, total-variation Lipschitz stability, a quantitative weak-star modulus $\sqrt{d\log(e+d^{-1})}$, and attractor upper semicontinuity in both history spaces. Quartic coercivity permits arbitrary fixed bounds on kernel total variation.
\end{enumerate}
The quantitative kernel estimate relies on a regularizing effect in the delay variable: after integration against the heat semigroup, even a merely continuous bounded trajectory produces a logarithmic Lipschitz modulus. This avoids imposing a uniform temporal modulus on the initial histories.

\section{Functional setting and assumptions}\label{sec:setting}

Set
\[
 H=L^2(\Omega),\qquad V=H_0^1(\Omega),\qquad V'=H^{-1}(\Omega),
 \qquad \norm{v}_V=\norm{\nabla v}_{L^2}.
\]
We use the Gelfand triple $V\hookrightarrow H\hookrightarrow V'$. Let $A$ be the positive self-adjoint Dirichlet Laplacian on $H$, with
\[
 D(A)=\{v\in H_0^1(\Omega):\Delta v\in H\}.
\]
Its inverse is compact, $D(A^{1/2})=V$, and
\begin{equation}\label{eq:poincare}
 \norm{v}_V^2\geq\lambda_1\norm{v}_H^2,\qquad
 \norm{q}_{V'}\leq\lambda_1^{-1/2}\norm{q}_H,
\end{equation}
where $\lambda_1>0$ is the first eigenvalue. The same symbol $A:V\to V'$ denotes the variational operator
$\dual{Av}{w}=(\nabla v,\nabla w)_{L^2}$.

Fix $r>0$, write $I=[-r,0]$, and define
\[
 X=C(I;H),\qquad
 \norm{\varphi}_X=\sup_{\theta\in I}\norm{\varphi(\theta)}_H.
\]
For $u:[-r,T]\to H$, its history segment is $u_t(\theta)=u(t+\theta)$; time derivatives are denoted by $\partial_tu$.
Let $\M(I)=C(I)'$ be the space of finite signed Borel measures, with
$\norm{\mu}_{\mathrm{TV}}=|\mu|(I)$.

\begin{assumption}\label{ass:main}
The forcing $h\in V'$ is independent of time. The maps $f,B:H\to H$ are Lipschitz on bounded balls: for every $R>0$ there are finite constants $L_f(R),L_B(R)$ such that
\[
 \begin{aligned}
 \norm{f(u)-f(v)}_H&\leq L_f(R)\norm{u-v}_H,\\
 \norm{B(u)-B(v)}_H&\leq L_B(R)\norm{u-v}_H
 \end{aligned}
\]
when $\norm{u}_H,\norm{v}_H\leq R$.
There are constants $b_0,b_1,c_0\geq0$ and $a_f\in\R$ such that
\begin{equation}\label{eq:structural}
 \norm{B(u)}_H\leq b_0+b_1\norm{u}_H,\qquad
 (f(u),u)_H\geq a_f\norm{u}_H^2-c_0,\qquad u\in H.
\end{equation}
\end{assumption}

For $\mu\in\M(I)$ and $u\in C([-r,T];H)$, define
\begin{equation}\label{eq:delayop}
 \calD_\mu u(t)=\int_I B(u(t+\theta))\,d\mu(\theta).
\end{equation}
The integrand is continuous and bounded in $H$. The Bochner integral is taken against $|\mu|$, with the sign specified by the Jordan decomposition. The equation is
\begin{equation}\label{eq:main}
 \partial_tu+Au+f(u)=\calD_\mu u+h,\qquad
 u|_{[-r,0]}=\varphi\in X.
\end{equation}

\begin{remark}[Scope of the reaction hypothesis]\label{rem:scope}
No upper growth bound on $f$ is imposed. Once an $H$-bound $R$ is known,
\begin{equation}\label{eq:fball}
 \norm{f(u)}_H\leq \norm{f(0)}_H+L_f(R)R
 \quad\text{for }\norm{u}_H\leq R.
\end{equation}
Similarly, bounded-ball Lipschitz continuity gives
$(f(u)-f(v),u-v)_H\geq-L_f(R)\norm{u-v}_H^2$
on that ball, which suffices for uniqueness.

These are operator hypotheses on $H$, not polynomial Nemytskii hypotheses. On a nonatomic domain, bounded-ball Lipschitz continuity of a scalar Nemytskii map $u\mapsto g(u)$ on $L^2$ forces $g$ to be globally Lipschitz: apply the bound on one fixed ball to $u=a\mathbf1_E$ and $v=b\mathbf1_E$, choosing $E$ small enough, and cancel $|E|^{1/2}$. Pointwise cubic reactions are treated separately in Section~\ref{sec:cubic}, with a different weak-solution regularity class.
\end{remark}

\begin{example}[Bounded feedback and superlinear abstract reactions]\label{ex:admissible}
A bounded feedback, such as $B(u)(x)=\tanh(u(x))$, admits $b_1=0$ and $b_0=|\Omega|^{1/2}$. More generally, globally Lipschitz scalar feedbacks define admissible Nemytskii maps. The abstract class also contains feedbacks
$B(u)=\rho(\norm{u}_H)u$ with bounded locally Lipschitz $\rho$.

The reaction $f(u)=\norm{u}_H^2u-au$, $a\in\R$, is Lipschitz on bounded $H$-balls and has superlinear $H$-growth. For every $k\geq0$,
\[
 (f(u),u)_H\geq k\norm{u}_H^2-\frac{(k+a)_+^2}{4},
 \qquad x_+=\max\{x,0\}.
\]
Consequently its coercivity coefficient can be chosen as large as required, with a corresponding change in the additive constant.
\end{example}

\begin{definition}[Weak solution]
A weak solution on $[0,T]$ is a function
\[
 u\in C([-r,T];H)\cap L^2(0,T;V),\qquad
 \partial_tu\in L^2(0,T;V'),
\]
equal to $\varphi$ on $[-r,0]$, such that for every $v\in V$ and almost every $t\in(0,T)$,
\begin{equation}\label{eq:weakform}
 \dual{\partial_tu}{v}+(\nabla u,\nabla v)_{L^2}
 +(f(u),v)_H=(\calD_\mu u,v)_H+\dual{h}{v}.
\end{equation}
\end{definition}

\section{Delay estimates and measure convergence}\label{sec:delay}

\begin{lemma}[Delay bounds]\label{lem:delaybounds}
Under Assumption~\ref{ass:main},
\begin{equation}\label{eq:delaygrowth}
 \norm{\calD_\mu u(t)}_H
 \leq\norm{\mu}_{\mathrm{TV}}
 \left(b_0+b_1\sup_{s\in[t-r,t]}\norm{u(s)}_H\right).
\end{equation}
If $u,v$ are bounded by $R$ in $C([-r,T];H)$, then
\begin{equation}\label{eq:delaylip}
 \norm{\calD_\mu u(t)-\calD_\mu v(t)}_H
 \leq\norm{\mu}_{\mathrm{TV}}L_B(R)
 \sup_{s\in[t-r,t]}\norm{u(s)-v(s)}_H,
\end{equation}
and
\begin{equation}\label{eq:delayLtwo}
 \norm{\calD_\mu u-\calD_\mu v}_{L^2(0,T;H)}
 \leq\norm{\mu}_{\mathrm{TV}}L_B(R)
 \norm{u-v}_{L^2(-r,T;H)}.
\end{equation}
\end{lemma}
\begin{proof}
The first two bounds follow from total variation. For the last one, the case $\mu=0$ is immediate. Otherwise, Cauchy--Schwarz with respect to $|\mu|$ and Fubini give
\begin{align*}
 &\int_0^T\norm{\calD_\mu u(t)-\calD_\mu v(t)}_H^2\,dt\\
 &\quad\leq\norm{\mu}_{\mathrm{TV}}L_B(R)^2
 \int_I\int_0^T\norm{u(t+\theta)-v(t+\theta)}_H^2\,dt\,d|\mu|(\theta)\\
 &\quad\leq\norm{\mu}_{\mathrm{TV}}^2L_B(R)^2
 \norm{u-v}_{L^2(-r,T;H)}^2.
\end{align*}
The same argument applies to bounded measurable trajectories, with the delay integral interpreted for almost every time.
\end{proof}

\begin{lemma}[Delay Gronwall inequality]\label{lem:delaygronwall}
Let $y\in C([-r,T];[0,\infty))$ satisfy
\[
 y(t)\leq a+b\int_0^t\sup_{\sigma\in[s-r,s]}y(\sigma)\,ds,
 \qquad a,b\geq0.
\]
Then
\[
 \sup_{s\in[-r,t]}y(s)
 \leq\max\{\sup_{s\in[-r,0]}y(s),a\}e^{bt},
 \qquad 0\leq t\leq T.
\]
\end{lemma}
\begin{proof}
Set $Y(t)=\sup_{s\in[-r,t]}y(s)$. Taking the supremum in the assumed inequality gives
$Y(t)\leq\max\{Y(0),a\}+b\int_0^tY(s)\,ds$.
Apply the ordinary Gronwall inequality.
\end{proof}

We use the bounded-Lipschitz distance
\begin{equation}\label{eq:BL}
 \dBL(\mu,\nu)=
 \sup_{\norm{q}_\infty+\Lip(q)\leq1}
 \left|\int_Iq\,d(\mu-\nu)\right|,
\end{equation}
where the test functions are real-valued. We also fix the increasing modulus
\begin{equation}\label{eq:modulus}
 \varpi(0)=0,\qquad
 \varpi(d)=d\log(e+d^{-1}),\quad d>0.
\end{equation}
In particular, $\varpi(d)\to0$ as $d\downarrow0$, $\varpi(d)\geq d$, and $\varpi(d)=O(d^\alpha)$ near zero for every $0<\alpha<1$.

\begin{lemma}[Metrization on bounded measure classes]\label{lem:BLweakstar}
On $\{\mu:\norm{\mu}_{\mathrm{TV}}\leq M\}$, the metric $\dBL$ induces the weak-star topology.
\end{lemma}
\begin{proof}
The unit ball of scalar bounded-Lipschitz functions is compact in $C(I)$ by Arzel\`a--Ascoli. Weak-star convergence with bounded total variation is uniform on this ball, by a finite-net argument, and hence implies convergence in $\dBL$. Conversely, Lipschitz functions are uniformly dense in $C(I)$. Approximate a continuous test function by a Lipschitz function and use the common total-variation bound for the approximation error.
\end{proof}

\begin{lemma}[Vector-valued weak-star convergence]\label{lem:weakstaruniform}
Suppose $\mu_n\weakstar\mu$ in $\M(I)$. If $Q\subset C(I;H)$ is compact, then
\begin{equation}\label{eq:uniformcompactmeasure}
 \sup_{g\in Q}\left\|\int_Ig\,d(\mu_n-\mu)\right\|_H\longrightarrow0.
\end{equation}
In particular, for a fixed $u\in C([-r,T];H)$,
\begin{equation}\label{eq:fixedtrajectoryweakstar}
 \sup_{0\leq t\leq T}
 \left\|\int_I B(u(t+\theta))\,d(\mu_n-\mu)(\theta)\right\|_H
 \longrightarrow0.
\end{equation}
\end{lemma}
\begin{proof}
Uniform boundedness gives $\sup_n\norm{\mu_n}_{\mathrm{TV}}<\infty$. For fixed $g\in C(I;H)$, piecewise-linear interpolation on a fine partition approximates $g$ uniformly by a finite sum $\sum_jq_jv_j$, with scalar continuous $q_j$ and vectors $v_j\in H$. Scalar weak-star convergence gives norm convergence of the integrals of this finite sum; the common total-variation bound controls the remainder. The operators $g\mapsto\int_Ig\,d(\mu_n-\mu)$ have uniformly bounded norms, so a finite net in $Q$ proves \eqref{eq:uniformcompactmeasure}. Finally, the map $t\mapsto[\theta\mapsto B(u(t+\theta))]$ is continuous from $[0,T]$ into $C(I;H)$ and therefore has compact image.
\end{proof}

\section{Global weak well-posedness}\label{sec:wellposed}

\begin{theorem}\label{thm:wellposed}
Under Assumption~\ref{ass:main}, every $\varphi\in X$ and $\mu\in\M(I)$ determine a unique global weak solution of \eqref{eq:main}. For every $T>0$,
\begin{equation}\label{eq:finiteapriori}
 \sup_{t\in[-r,T]}\norm{u(t)}_H^2+
 \int_0^T\norm{u(t)}_V^2\,dt\leq C_T.
\end{equation}
The bound is uniform when $\varphi$ ranges over a bounded subset of $X$ and $\mu$ ranges over a total-variation-bounded subset of $\M(I)$.
\end{theorem}
\begin{proof}
Let $\{e_j\}_{j\geq1}$ be an orthonormal eigenbasis of $A$, and let $P_m$ be the projection onto $H_m=\operatorname{span}\{e_1,\dots,e_m\}$. Compactness of $\varphi(I)$ and strong convergence $P_m\to I$ imply
\begin{equation}\label{eq:historyprojection}
 \norm{P_m\varphi-\varphi}_{C(I;H)}\longrightarrow0.
\end{equation}
Prescribe $u_m=P_m\varphi$ on $I$ and solve
\begin{equation}\label{eq:galerkin}
 (\partial_tu_m,e_j)_H+(Au_m,e_j)_H+(f(u_m),e_j)_H
 =(\calD_\mu u_m,e_j)_H+\dual{h}{e_j},\quad j\leq m.
\end{equation}
The finite-dimensional history functional is Lipschitz on bounded sets. An atom at zero contributes the instantaneous term $\mu(\{0\})P_mB(u_m(t))$. Standard functional differential equation theory therefore gives a local solution.

Testing by $u_m$, applying \eqref{eq:structural} and \eqref{eq:delaygrowth}, and using Young's inequality yields
\begin{equation}\label{eq:finiteenergy}
 \frac{d}{dt}\norm{u_m(t)}_H^2+\norm{u_m(t)}_V^2
 \leq C\left(1+\sup_{s\in[t-r,t]}\norm{u_m(s)}_H^2\right).
\end{equation}
Here $C$ depends on the coercivity and feedback constants, $\norm{h}_{V'}$, and $\norm{\mu}_{\mathrm{TV}}$, but not on $m$. Lemma~\ref{lem:delaygronwall} bounds the $H$-norm on each finite interval, and integration of \eqref{eq:finiteenergy} then bounds the $L^2(0,T;V)$ norm. These estimates prevent finite-time blow-up.

Let $R_T$ be a common $H$-bound. Equation \eqref{eq:fball} bounds $f(u_m)$ in $L^\infty(0,T;H)$ without an upper growth assumption on $f$. Lemma~\ref{lem:delaybounds} gives the same bound for $\calD_\mu u_m$. The spectral projections are contractions on $V$ and $V'$, and
\[
 h_m=\sum_{j=1}^m\dual{h}{e_j}e_j
 \longrightarrow h\quad\text{in }V'.
\]
Thus
$\partial_tu_m=-Au_m-P_mf(u_m)+P_m\calD_\mu u_m+h_m$
is bounded in $L^2(0,T;V')$.

After passage to a subsequence,
\[
 \begin{split}
 u_m&\rightharpoonup u \quad\text{in }L^2(0,T;V),\\
 u_m&\rightharpoonup^*u\quad\text{in }L^\infty(0,T;H),\\
 \partial_tu_m&\rightharpoonup\partial_tu
       \quad\text{in }L^2(0,T;V').
 \end{split}
\]
The Aubin--Lions--Simon lemma \cite{simon1987} gives strong convergence in $L^2(0,T;H)$. Lions' continuity lemma gives $u\in C([0,T];H)$. The trace map at zero is continuous from
$L^2(0,T;V)\cap H^1(0,T;V')$ to $H$, so
$u_m(0)\rightharpoonup u(0)$ in $H$. Since $u_m(0)=P_m\varphi(0)\to\varphi(0)$ strongly, $u(0)=\varphi(0)$. Joining $u$ to $\varphi$ therefore gives a continuous function on $[-r,T]$.

By \eqref{eq:historyprojection}, the joined trajectories converge strongly in $L^2(-r,T;H)$. Bounded-ball Lipschitz continuity gives
$f(u_m)\to f(u)$ in $L^2(0,T;H)$, and \eqref{eq:delayLtwo} gives
$\calD_\mu u_m\to\calD_\mu u$ in that space. We can now pass to the limit in \eqref{eq:galerkin}, first against finite eigenfunction combinations and then against arbitrary $v\in V$. Lower semicontinuity proves \eqref{eq:finiteapriori}.

For uniqueness, let $w=u-v$ be the difference of two solutions and let $R_T$ bound their $H$-norms. The energy identity and the local bounds imply
\begin{equation}\label{eq:uniquenessenergy}
 \frac{d}{dt}\norm{w(t)}_H^2+\norm{w(t)}_V^2
 \leq C_T\sup_{s\in[t-r,t]}\norm{w(s)}_H^2.
\end{equation}
For equal histories, Lemma~\ref{lem:delaygronwall} gives $w=0$. Solutions constructed on different finite intervals agree by uniqueness, defining the global solution.
\end{proof}

\begin{corollary}[Continuous semiflow]\label{cor:initialdata}
Define $S_\mu(t)\varphi=u_t$. These maps form a jointly continuous semiflow on $X$. For each bounded $D\subset X$, $T>0$, and $M\geq0$,
\begin{equation}\label{eq:initialdata}
 \sup_{0\leq t\leq T}\norm{S_\mu(t)\varphi-S_\mu(t)\psi}_X
 \leq C_{T,D,M}\norm{\varphi-\psi}_X
\end{equation}
for $\varphi,\psi\in D$ and $\norm{\mu}_{\mathrm{TV}}\leq M$.
\end{corollary}
\begin{proof}
The finite-time bound is uniform over these histories and measures. Apply \eqref{eq:uniquenessenergy} and Lemma~\ref{lem:delaygronwall} with the nonzero initial difference. The history norm is bounded by $\sup_{s\in[-r,T]}\norm{u^\varphi(s)-u^\psi(s)}_H$. Uniqueness and time translation give the semiflow identity. For fixed $\varphi$, continuity of $u$ gives continuity of $t\mapsto u_t$ in $X$; combining this with the locally uniform estimate \eqref{eq:initialdata} proves joint continuity.
\end{proof}

\section{Quantitative stability with respect to the delay measure}\label{sec:stability}

For trajectories with the weak-solution regularity, write
\[
 \norm{w}_{\mathcal E_T}
 =\norm{w}_{C([-r,T];H)}
  +\norm{w}_{L^2(0,T;V)}
  +\norm{\partial_tw}_{L^2(0,T;V')}.
\]
We first record an energy estimate that also gives strong convergence in this norm.

\begin{lemma}[Stability through a reference residual]\label{lem:residual}
Fix $T>0$, a bounded set $D\subset X$, and $M\geq0$. Let $u_\mu^\varphi,u_\nu^\psi$ be solutions with $\varphi,\psi\in D$ and $\norm{\mu}_{\mathrm{TV}},\norm{\nu}_{\mathrm{TV}}\leq M$. Define
\begin{equation}\label{eq:residual}
 \rho_{\mu,\nu}(t)
 =\int_I B(u_\nu^\psi(t+\theta))\,d(\mu-\nu)(\theta).
\end{equation}
Then
\begin{equation}\label{eq:residualestimate}
 \norm{u_\mu^\varphi-u_\nu^\psi}_{\mathcal E_T}
 \leq C_{T,D,M}
 \left(\norm{\varphi-\psi}_X+
       \norm{\rho_{\mu,\nu}}_{L^2(0,T;H)}\right).
\end{equation}
\end{lemma}
\begin{proof}
Let $w=u_\mu^\varphi-u_\nu^\psi$. Subtract the equations and split the delay difference into
$\calD_\mu u_\mu^\varphi-\calD_\mu u_\nu^\psi+\rho_{\mu,\nu}$.
On the common finite-time $H$-ball, testing by $w$ and using Young's inequality gives
\[
 \frac{d}{dt}\norm{w(t)}_H^2+\norm{w(t)}_V^2
 \leq C_{T,D,M}\sup_{s\in[t-r,t]}\norm{w(s)}_H^2
       +C\norm{\rho_{\mu,\nu}(t)}_H^2.
\]
Integration and Gronwall bound both the supremum norm and the $L^2(0,T;V)$ norm of $w$. In the difference equation, $Aw$ is controlled in $L^2(0,T;V')$ by the latter norm. The reaction difference and the first delay difference are controlled in $L^2(0,T;H)$ by the supremum norm, while the residual is already in that space. The embedding in \eqref{eq:poincare} therefore bounds $\partial_tw$ and proves \eqref{eq:residualestimate}.
\end{proof}

\begin{theorem}[Total-variation Lipschitz stability]\label{thm:tvcontinuous}
Under the hypotheses of Lemma~\ref{lem:residual},
\begin{equation}\label{eq:tvrate}
 \norm{u_\mu^\varphi-u_\nu^\psi}_{\mathcal E_T}
 \leq C_{T,D,M}
 \left(\norm{\varphi-\psi}_X+\norm{\mu-\nu}_{\mathrm{TV}}\right).
\end{equation}
In particular, the same bound holds for
$\sup_{0\leq t\leq T}\norm{S_\mu(t)\varphi-S_\nu(t)\psi}_X$.
\end{theorem}
\begin{proof}
If $R_T$ is the common trajectory bound, then
$\norm{\rho_{\mu,\nu}(t)}_H
\leq(b_0+b_1R_T)\norm{\mu-\nu}_{\mathrm{TV}}$.
Apply Lemma~\ref{lem:residual}.
\end{proof}

Total variation does not capture the approximation of a Dirac measure by an absolutely continuous kernel. To quantify such limits for rough histories, we estimate the perturbation after integration against the heat semigroup.

\begin{lemma}[Regularity in the delay variable]\label{lem:averaging}
Let $E(t)=e^{-tA}$, $g\in C([-r,T];H)$, and $0\leq t\leq T$. Set
\begin{equation}\label{eq:averaged}
 G_{t,g}(\theta)=\int_0^t E(t-s)g(s+\theta)\,ds,\qquad \theta\in I.
\end{equation}
There is $C_A>0$, independent of $T,t,g$, such that
\begin{equation}\label{eq:Gbounds}
 \norm{G_{t,g}}_{C(I;H)}\leq\lambda_1^{-1}\norm{g}_\infty,\qquad
 \norm{G_{t,g}(\theta)-G_{t,g}(\eta)}_H
 \leq C_A\norm{g}_\infty\varpi(|\theta-\eta|).
\end{equation}
Here $\norm{g}_\infty=\norm{g}_{C([-r,T];H)}$.
\end{lemma}
\begin{proof}
The spectral theorem gives
\begin{equation}\label{eq:analyticbounds}
 \norm{E(s)}_{\mathcal L(H)}\leq e^{-\lambda_1s},\qquad
 \norm{AE(s)}_{\mathcal L(H)}
 \leq Cs^{-1}e^{-\lambda_1s/2},\quad s>0.
\end{equation}
The first estimate proves the supremum bound. Moreover, integration of the second estimate, together with the direct bound on $E$, gives
\begin{equation}\label{eq:semigroupshift}
 \norm{E(s+a)-E(s)}_{\mathcal L(H)}
 \leq C\min\{1,a/s\}e^{-\lambda_1s/2},
 \qquad a,s>0.
\end{equation}

Take $\eta=\theta+a$, $a>0$, and change variables in \eqref{eq:averaged}:
\[
 G_{t,g}(\theta)
 =\int_\theta^{t+\theta}E(t+\theta-q)g(q)\,dq.
\]
If $a\geq t$, each integral has norm at most $t\norm{g}_\infty$, giving a bound $2a\norm{g}_\infty$. If $a<t$, the two nonoverlapping endpoint intervals contribute at most $2a\norm{g}_\infty$. On the overlap, \eqref{eq:semigroupshift} gives
\begin{align*}
 &\norm{G_{t,g}(\theta+a)-G_{t,g}(\theta)}_H\\
 &\quad\leq
 2a\norm{g}_\infty+
 C\norm{g}_\infty\int_0^\infty
       \min\{1,a/s\}e^{-\lambda_1s/2}\,ds.
\end{align*}
For $0<a\leq1$, splitting the integral at $a$ and $1$ bounds it by
$C_Aa(1+\log(1/a))$. For $a>1$ the integral is bounded by a constant depending on $\lambda_1$. Both cases give the second estimate in \eqref{eq:Gbounds}.
\end{proof}

\begin{lemma}[Integrated kernel perturbation]\label{lem:integratedkernel}
For every $M\geq0$ there is $C_{A,r,M}>0$ such that, if $\norm{\mu}_{\mathrm{TV}},\norm{\nu}_{\mathrm{TV}}\leq M$, then
\begin{equation}\label{eq:integratedkernel}
 \sup_{0\leq t\leq T}
 \left\|\int_I G_{t,g}(\theta)\,d(\mu-\nu)(\theta)\right\|_H
 \leq C_{A,r,M}\norm{g}_\infty\varpi(\dBL(\mu,\nu)).
\end{equation}
The constant is independent of $T$.
\end{lemma}
\begin{proof}
Write $d=\dBL(\mu,\nu)$. The case $d=0$ is immediate. First suppose
$0<d\leq d_0:=\min\{1,r\}$. Divide $I$ into $N=\lceil r/d\rceil$ equal intervals, of length $\ell\in[d/2,d]$, and let $G^\ell$ be the piecewise-linear interpolant of $G=G_{t,g}$. Lemma~\ref{lem:averaging} gives
\[
 \norm{G-G^\ell}_\infty\leq C_A\norm{g}_\infty\varpi(d),
 \qquad
 \norm{G^\ell}_\infty\leq\lambda_1^{-1}\norm{g}_\infty,
\]
and
\[
 \Lip_H(G^\ell)
 \leq C_A\norm{g}_\infty\frac{\varpi(\ell)}{\ell}
 \leq C_A\norm{g}_\infty\log(e+2/d).
\]
By testing against unit vectors in $H$, scalar bounded-Lipschitz duality implies
\[
 \left\|\int_I G^\ell\,d(\mu-\nu)\right\|_H
 \leq d\left(\norm{G^\ell}_\infty+\Lip_H(G^\ell)\right).
\]
The interpolation remainder is bounded by $2M\norm{G-G^\ell}_\infty$. Since
$d\log(e+2/d)\leq C\varpi(d)$, this proves the estimate for $d\leq d_0$. For $d>d_0$, use
$\norm{\int_I G\,d(\mu-\nu)}_H\leq2M\lambda_1^{-1}\norm{g}_\infty$
and $\varpi(d)\geq\varpi(d_0)>0$.
\end{proof}

\begin{theorem}[Quantitative weak-star stability]\label{thm:BLcontinuous}
Let Assumption~\ref{ass:main} hold. For every $T>0$, bounded $D\subset X$, and $M\geq0$, there is $C_{T,D,M}$ such that
\begin{equation}\label{eq:BLrate}
 \sup_{0\leq t\leq T}
 \norm{S_\mu(t)\varphi-S_\nu(t)\psi}_X
 \leq C_{T,D,M}
 \left(\norm{\varphi-\psi}_X+\varpi(\dBL(\mu,\nu))\right)
\end{equation}
whenever $\varphi,\psi\in D$ and $\norm{\mu}_{\mathrm{TV}},\norm{\nu}_{\mathrm{TV}}\leq M$.
No temporal modulus of continuity is prescribed for histories in $D$.
\end{theorem}
\begin{proof}
Let $u=u_\mu^\varphi$, $v=u_\nu^\psi$, $w=u-v$, and let $R_T$ be a common $H$-bound. Although the common forcing $h$ may lie only in $V'$, it cancels from the difference equation, whose remaining inhomogeneity is continuous with values in $H$. Hence its variation-of-constants formula is
\begin{align}\label{eq:milddifference}
 w(t)={}&E(t)w(0)
 -\int_0^t E(t-s)[f(u(s))-f(v(s))]\,ds \notag\\
 &+\int_0^t E(t-s)
       [\calD_\mu u(s)-\calD_\mu v(s)]\,ds
 +R_{\mu,\nu}(t),
\end{align}
where
\[
 R_{\mu,\nu}(t)
 =\int_0^t E(t-s)\int_I B(v(s+\theta))\,d(\mu-\nu)(\theta)\,ds.
\]
Fubini identifies this term with the left-hand side of
\eqref{eq:integratedkernel}, taking $g=B(v)$, so
\begin{equation}\label{eq:Rmild}
 \sup_{t\leq T}\norm{R_{\mu,\nu}(t)}_H
 \leq C_{A,r,M}(b_0+b_1R_T)\varpi(\dBL(\mu,\nu)).
\end{equation}

Set $Z(t)=\sup_{s\in[-r,t]}\norm{w(s)}_H$ and
$L_T=L_f(R_T)+M L_B(R_T)$.
The contraction property of $E$, the local Lipschitz estimates, and \eqref{eq:milddifference} give
\[
 Z(t)\leq\norm{\varphi-\psi}_X
 +L_T\int_0^tZ(s)\,ds
 +C_{A,r,M}(b_0+b_1R_T)\varpi(\dBL(\mu,\nu)).
\]
Gronwall proves \eqref{eq:BLrate}, since the history difference at every $t\leq T$ is bounded by $Z(T)$.
\end{proof}

\begin{remark}\label{rem:rateinterpretation}
The estimate concerns the solution histories after the perturbation has been propagated by the equation. It does not assert a uniform bounded-Lipschitz estimate for the instantaneous map
$\varphi\mapsto\int_I B(\varphi(\theta))\,d\mu(\theta)$
on bounded subsets of $X$. Lemma~\ref{lem:averaging} supplies the temporal averaging needed for the solution estimate. The argument gives an upper bound; no optimality claim is made for the logarithmic factor.
\end{remark}

\begin{corollary}[Weak-star convergence and energy convergence]\label{thm:weakstarcontinuous}
Suppose $\mu_n\weakstar\mu$ in $\M(I)$. For every bounded $D\subset X$ and $T>0$,
\begin{equation}\label{eq:weakstarbounded}
 \sup_{\varphi\in D}\sup_{0\leq t\leq T}
 \norm{S_{\mu_n}(t)\varphi-S_\mu(t)\varphi}_X
 \longrightarrow0.
\end{equation}
If also $\varphi_n\to\varphi$ in $X$, then the corresponding trajectories satisfy
\begin{equation}\label{eq:weakstarenergy}
 \norm{u_{\mu_n}^{\varphi_n}-u_\mu^\varphi}_{\mathcal E_T}
 \longrightarrow0.
\end{equation}
\end{corollary}
\begin{proof}
The total variations are uniformly bounded by the uniform boundedness principle. Lemma~\ref{lem:BLweakstar} and Theorem~\ref{thm:BLcontinuous} prove \eqref{eq:weakstarbounded}. For the stronger energy statement, use the fixed reference solution $u_\mu^\varphi$ in \eqref{eq:residual}. Its residual tends to zero uniformly on $[0,T]$ by Lemma~\ref{lem:weakstaruniform}. Lemma~\ref{lem:residual} now proves \eqref{eq:weakstarenergy}.
\end{proof}

\begin{corollary}[Concentration to a discrete delay]\label{cor:distributedtodiscrete}
Let $\tau\in[0,r]$ and suppose $K_n(\theta)\,d\theta\weakstar\alpha\delta_{-\tau}$ in $\M(I)$, with $K_n\in L^1(-r,0)$. The corresponding solutions converge as in \eqref{eq:weakstarbounded} and \eqref{eq:weakstarenergy}.

More quantitatively, let $\rho_\varepsilon\geq0$ be an $L^1(I)$ density with integral one and
\[
 \int_I|\theta+\tau|\rho_\varepsilon(\theta)\,d\theta
 \leq C_\rho\varepsilon.
\]
Set $K_\varepsilon=\alpha\rho_\varepsilon$ and $\mu_\varepsilon=K_\varepsilon(\theta)\,d\theta$. Then, uniformly for $\varphi$ in bounded subsets of $X$,
\begin{equation}\label{eq:concentrationrate}
 \sup_{0\leq t\leq T}
 \norm{S_{\mu_\varepsilon}(t)\varphi
       -S_{\alpha\delta_{-\tau}}(t)\varphi}_X
 \leq C_{T,D,\alpha,\rho}\,
       \varepsilon\log(e+\varepsilon^{-1})
\end{equation}
for sufficiently small $\varepsilon>0$.
\end{corollary}
\begin{proof}
The qualitative conclusion follows from Corollary~\ref{thm:weakstarcontinuous}; weak-star convergence already implies a uniform bound on $\norm{K_n}_{L^1}$. For every scalar test function in the unit bounded-Lipschitz ball,
\[
 \left|\int_I q(\theta)K_\varepsilon(\theta)\,d\theta
             -\alpha q(-\tau)\right|
 \leq|\alpha|\int_I|\theta+\tau|\rho_\varepsilon(\theta)\,d\theta.
\]
Thus $\dBL(\mu_\varepsilon,\alpha\delta_{-\tau})
\leq|\alpha|C_\rho\varepsilon$, and both total variations equal $|\alpha|$. Apply Theorem~\ref{thm:BLcontinuous}. A nonnegative smooth approximate identity centered at $-\tau$ has the required first-moment bound; at $\tau=0$ or $\tau=r$, choose a one-sided approximate identity.
\end{proof}

\begin{corollary}[Atomic quadrature of signed kernels]\label{cor:quadrature}
Let $\{I_j\}_{j=1}^N$ be a Borel partition of $I$ with $\operatorname{diam}(I_j)\leq\Delta$, and choose $\theta_j\in I_j$. Define
\[
 Q_\Delta\mu=\sum_{j=1}^N\mu(I_j)\delta_{\theta_j}.
\]
Then
\begin{equation}\label{eq:quadraturemeasure}
 \norm{Q_\Delta\mu}_{\mathrm{TV}}\leq\norm{\mu}_{\mathrm{TV}},
 \qquad
 \dBL(Q_\Delta\mu,\mu)\leq\Delta\norm{\mu}_{\mathrm{TV}}.
\end{equation}
For every $T>0$, bounded $D\subset X$, and $M\geq0$,
\begin{equation}\label{eq:quadraturerate}
 \sup_{\norm{\mu}_{\mathrm{TV}}\leq M}
 \sup_{\varphi\in D}\sup_{0\leq t\leq T}
 \norm{S_{Q_\Delta\mu}(t)\varphi-S_\mu(t)\varphi}_X
 \leq C_{T,D,M}\Delta\log(e+\Delta^{-1})
\end{equation}
for $0<\Delta\leq1$.
\end{corollary}
\begin{proof}
The variation bound follows from
$\sum_j|\mu(I_j)|\leq|\mu|(I)$. For a test function with $\Lip(q)\leq1$,
\[
 \left|\int_Iq\,d(Q_\Delta\mu-\mu)\right|
 =\left|\sum_j\int_{I_j}[q(\theta_j)-q(\theta)]\,d\mu(\theta)\right|
 \leq\Delta\norm{\mu}_{\mathrm{TV}}.
\]
Theorem~\ref{thm:BLcontinuous} gives \eqref{eq:quadraturerate}, using
$\varpi(M\Delta)\leq C_M\varpi(\Delta)$.
\end{proof}

\section{Uniform dissipativity and compact global attractors}\label{sec:dissipativity}

Fix $M_0\geq0$ and set
\[
 \Kclass_{M_0}=\{\mu\in\M(I):\norm{\mu}_{\mathrm{TV}}\leq M_0\}.
\]

\begin{assumption}[Damping dominance]\label{ass:dissipative}
In addition to Assumption~\ref{ass:main}, assume
\begin{equation}\label{eq:dissipativecondition}
 \lambda_1+a_f>M_0b_1.
\end{equation}
\end{assumption}

Only the growth slope of the feedback enters this condition. The intercept $b_0$ affects the absorbing radius. The forcing remains in $V'$.

\begin{lemma}[Halanay estimate]\label{lem:halanay}
Let $a>b\geq0$ and $c\geq0$. Suppose $y\in C([-r,\infty);[0,\infty))$ is locally absolutely continuous on $[0,\infty)$ and
\[
 y'(t)+ay(t)\leq b\sup_{s\in[t-r,t]}y(s)+c
 \quad\text{for almost every }t>0.
\]
Set $K=c/(a-b)$ and choose $\omega>0$ with
$\omega+be^{\omega r}<a$. Then
\begin{equation}\label{eq:halanay}
 y(t)\leq K+
 \left(\sup_{s\in[-r,0]}y(s)-K\right)_+e^{-\omega t},
 \qquad t\geq0.
\end{equation}
\end{lemma}
\begin{proof}
Such an $\omega$ exists because $a>b$. With
$L=(\sup_{[-r,0]}y-K)_+$, the function
$q(t)=K+Le^{-\omega t}$, defined also for $-r\leq t<0$, dominates the prescribed history. It satisfies
\[
 q'(t)+aq(t)-b\sup_{s\in[t-r,t]}q(s)-c
 =Le^{-\omega t}(a-\omega-be^{\omega r})\geq0.
\]
The scalar comparison argument, first with a strictly positive perturbation of this supersolution and then by passage to the limit, gives $y\leq q$. This is the usual Halanay argument \cite{halanay1966}; it also covers $b=0$.
\end{proof}

\begin{proposition}[Common invariant absorbing ball]\label{prop:Habsorbing}
Under Assumptions~\ref{ass:main} and \ref{ass:dissipative}, there is $R_H>0$ such that
\[
 \mathcal B_H=\{\varphi\in X:\norm{\varphi}_X\leq R_H\}
\]
is positively invariant for every $\mu\in\Kclass_{M_0}$. For every bounded $D\subset X$, there is $T_H(D)$, independent of $\mu$, such that
$S_\mu(t)D\subset\mathcal B_H$ for all $t\geq T_H(D)$.
\end{proposition}
\begin{proof}
Let $y(t)=\norm{u(t)}_H^2$, $Y(t)=\sup_{s\in[t-r,t]}y(s)$, and $k=M_0b_1$. The weak energy identity and coercivity give
\[
 \frac12y'(t)+\norm{u(t)}_V^2+a_fy(t)
 \leq c_0+M_0b_0\sqrt{y(t)}
          +k\sqrt{Y(t)y(t)}+\dual{h}{u(t)}.
\]
For $0<\eta<1$ and $\varepsilon>0$,
\[
 \begin{aligned}
 \dual{h}{u}&\leq\eta\norm{u}_V^2+\frac{\norm{h}_{V'}^2}{4\eta},\\
 M_0b_0\sqrt y&\leq\varepsilon y+\frac{(M_0b_0)^2}{4\varepsilon},
 \qquad
 k\sqrt{Yy}\leq\frac{k}{2}(Y+y).
 \end{aligned}
\]
Choose $\eta,\varepsilon$ so small that
$(1-\eta)\lambda_1+a_f-\varepsilon>k$, which is possible by
\eqref{eq:dissipativecondition}. It follows that
\begin{equation}\label{eq:dissipenergy}
 y'(t)+ay(t)\leq kY(t)+c,\qquad
 a=2[(1-\eta)\lambda_1+a_f-\varepsilon]-k>k,
\end{equation}
where $c$ is independent of $\mu$ and the initial history. Let $K=c/(a-k)$ and take $R_H^2=K+1$. Lemma~\ref{lem:halanay} proves positive invariance of the ball and a uniform eventual bound for $y$. Increasing the entry time by $r$ makes this bound valid on each entire history interval.
\end{proof}

\begin{proposition}[Uniform fractional smoothing with rough forcing]\label{prop:smoothing}
Under the same hypotheses, let $z\in V$ be the unique variational solution of $Az=h$. For every $0<s<1$ there is $R_s>0$ such that, for every bounded $D\subset X$,
\begin{equation}\label{eq:fractionalbound}
 \sup_{\mu\in\Kclass_{M_0}}\sup_{\varphi\in D}
 \norm{A^s(u_\mu^\varphi(t)-z)}_H\leq R_s
 \quad\text{for }t\geq T_H(D)+1.
\end{equation}
There are also constants $R_V,R_1$ independent of $D$ and $\mu$ such that, after this entry time,
\begin{equation}\label{eq:Vtimebound}
 \norm{u_\mu^\varphi(t)}_V\leq R_V,\qquad
 \norm{\partial_tu_\mu^\varphi(t)}_{V'}\leq R_1
\end{equation}
with the derivative estimate holding almost everywhere.
\end{proposition}
\begin{proof}
Existence of $z$ follows from the Lax--Milgram theorem, and
$\norm{z}_V\leq\norm{h}_{V'}$.
Put $w=u-z$. Then
\begin{equation}\label{eq:shifted}
 \partial_tw+Aw=F_\mu(t),\qquad
 F_\mu(t)=-f(u(t))+\calD_\mu u(t).
\end{equation}
For $t\geq T_H(D)$, Proposition~\ref{prop:Habsorbing} and
\eqref{eq:fball} give
\[
 \begin{aligned}
 \norm{w(t)}_H&\leq R_H+\norm{z}_H=:R_w,\\
 \norm{F_\mu(t)}_H&\leq
 \norm{f(0)}_H+L_f(R_H)R_H+M_0(b_0+b_1R_H)=:R_F.
 \end{aligned}
\]
Both bounds are uniform in $\mu$ and $D$ after entry. The linear variation-of-constants formula for \eqref{eq:shifted} gives
\[
 w(t)=E(1)w(t-1)+\int_0^1E(\sigma)F_\mu(t-\sigma)\,d\sigma.
\]
This identity follows from the scalar eigenfunction equations, or from the weak-to-mild equivalence for the linear parabolic equation. Analytic semigroup smoothing \cite{pazy1983} gives, for $0<s<1$,
\[
 \norm{A^sw(t)}_H
 \leq C_sR_w+C_sR_F\int_0^1\sigma^{-s}\,d\sigma
 =C_sR_w+\frac{C_sR_F}{1-s}.
\]
Taking $s=1/2$ yields $R_V=\norm{z}_V+R_{1/2}$. Finally,
\[
 \norm{\partial_tu(t)}_{V'}
 =\norm{-Aw(t)+F_\mu(t)}_{V'}
 \leq R_{1/2}+\lambda_1^{-1/2}R_F=:R_1,
\]
which includes the embedding constant from \eqref{eq:poincare}.
\end{proof}

\begin{proposition}[Common compact absorbing set]\label{prop:compactabsorbing}
There exists a compact set $\calK_*\subset\mathcal B_H\subset X$ such that, for every bounded $D\subset X$,
\[
 S_\mu(t)D\subset\calK_*
 \quad\text{for }t\geq T_H(D)+1+r,\quad
 \mu\in\Kclass_{M_0}.
\]
\end{proposition}
\begin{proof}
On an interval where \eqref{eq:Vtimebound} holds, the pivot-space interpolation inequality gives
\begin{align}\label{eq:timeholder}
 \norm{u(t)-u(s)}_H^2
 &\leq\norm{u(t)-u(s)}_V\norm{u(t)-u(s)}_{V'}\notag\\
 &\leq2R_VR_1|t-s|.
\end{align}
Let $C_H=(2R_VR_1)^{1/2}$. Define $\calK_*$ as the set of $\varphi\in X$ satisfying
\[
 \norm{\varphi}_X\leq R_H,\qquad
 \sup_{\theta\in I}\norm{\varphi(\theta)}_V\leq R_V,
\]
and
\[
 \norm{\varphi(\theta)-\varphi(\eta)}_H
 \leq C_H|\theta-\eta|^{1/2},\qquad \theta,\eta\in I.
\]
The $V$-ball is compact in $H$, so Arzel\`a--Ascoli gives relative compactness in $X$. These conditions are closed under uniform $H$-convergence: for the $V$-bound, use weak compactness in $V$ and lower semicontinuity at each time. Thus $\calK_*$ is compact. Propositions~\ref{prop:Habsorbing} and \ref{prop:smoothing}, together with \eqref{eq:timeholder}, place every sufficiently late history in this set.
\end{proof}

\begin{theorem}[Compact global attractors]\label{thm:attractorexistence}
For every $\mu\in\Kclass_{M_0}$, the semiflow $S_\mu$ has a compact global attractor $\calA_\mu\subset X$. It is strictly invariant and attracts every bounded subset of $X$. Moreover,
\begin{equation}\label{eq:commonattractor}
 \bigcup_{\mu\in\Kclass_{M_0}}\calA_\mu\subset\calK_*,
\end{equation}
and for every $0<s<1$,
\begin{equation}\label{eq:attractorfractional}
 \sup_{\mu\in\Kclass_{M_0}}\sup_{\varphi\in\calA_\mu}
 \sup_{\theta\in I}
 \norm{A^s(\varphi(\theta)-z)}_H\leq R_s.
\end{equation}
\end{theorem}
\begin{proof}
Corollary~\ref{cor:initialdata} supplies a continuous semiflow, and Proposition~\ref{prop:compactabsorbing} supplies compact absorption. The standard global-attractor theorem therefore applies; see \cite{hale1988,robinson2001}. Since $\calK_*$ absorbs the bounded set $\calA_\mu$, strict invariance gives
$\calA_\mu=S_\mu(T)\calA_\mu\subset\calK_*$ for large $T$.

For \eqref{eq:attractorfractional}, choose
$T\geq T_H(\calK_*)+1+r$. Every $\varphi\in\calA_\mu$ has a preimage
$\psi\in\calA_\mu\subset\calK_*$ under $S_\mu(T)$. Apply
\eqref{eq:fractionalbound} to the trajectory from $\psi$ at times $T+\theta$, $\theta\in I$.
\end{proof}

\begin{remark}[Consequences of separating growth and damping]
If $B$ is bounded and $\lambda_1+a_f>0$, the preceding results apply to every fixed class $\Kclass_{M_0}$, however large $M_0$ is. The radii and entry times can depend on $M_0$. Likewise, the superlinear abstract reaction in Example~\ref{ex:admissible} allows every such class, by choosing its coercivity coefficient sufficiently large. These statements do not require a uniform bound over all finite measures at once.
\end{remark}

\section{Upper semicontinuity in the state and energy spaces}\label{sec:usc}

For nonempty subsets $C,D$ of a normed space $Y$, write
\[
 \dist_Y(C,D)=\sup_{c\in C}\inf_{d\in D}\norm{c-d}_Y.
\]
The set $\calK_*$ below is the common compact absorbing set, and all hypotheses of Section~\ref{sec:dissipativity} remain in force.

\begin{theorem}[Attractor perturbation bound and upper semicontinuity]\label{thm:usc}
There are constants $C,L>0$, independent of $\mu,\nu\in\Kclass_{M_0}$, such that
\begin{equation}\label{eq:attractormodulus}
 \dist_X(\calA_\nu,\calA_\mu)
 \leq
 \inf_{T\geq0}
 \left\{Ce^{LT}\varpi(\dBL(\mu,\nu))
       +a_\mu(T)\right\},
\end{equation}
where
$a_\mu(T)=\dist_X(S_\mu(T)\calK_*,\calA_\mu)\to0$ as $T\to\infty$.
Consequently,
\begin{equation}\label{eq:usc}
 \mu_n\weakstar\mu,\quad \mu_n,\mu\in\Kclass_{M_0}
 \quad\Longrightarrow\quad
 \dist_X(\calA_{\mu_n},\calA_\mu)\longrightarrow0.
\end{equation}
\end{theorem}
\begin{proof}
Because $\calK_*\subset\mathcal B_H$ and $\mathcal B_H$ is invariant for all kernels, reference trajectories starting in $\calK_*$ have the same $H$-bound $R_H$ for all future times. Lemma~\ref{lem:integratedkernel} has a constant independent of the time horizon. The proof of Theorem~\ref{thm:BLcontinuous} therefore gives
\begin{equation}\label{eq:uniformexpfinite}
 \sup_{\varphi\in\calK_*}
 \norm{S_\nu(T)\varphi-S_\mu(T)\varphi}_X
 \leq Ce^{LT}\varpi(\dBL(\mu,\nu)),
 \qquad T\geq0,
\end{equation}
for example with $L=1+L_f(R_H)+M_0L_B(R_H)$.

Fix $T\geq0$. For every $a\in\calA_\nu$, invariance gives
$b\in\calA_\nu\subset\calK_*$ with $a=S_\nu(T)b$. Hence
\[
 \dist_X(a,\calA_\mu)
 \leq\norm{S_\nu(T)b-S_\mu(T)b}_X
       +\dist_X(S_\mu(T)b,\calA_\mu).
\]
Taking the supremum over $a$ and then the infimum over $T$ proves
\eqref{eq:attractormodulus}. Attraction of the bounded set $\calK_*$ gives $a_\mu(T)\to0$. For a weak-star convergent sequence, first choose $T$ to make this term small, then use Lemma~\ref{lem:BLweakstar} in the first term.
\end{proof}

\begin{corollary}[Conditional rate for global attractors]\label{cor:attractorrate}
Fix $\mu\in\Kclass_{M_0}$. If, in addition,
$a_\mu(T)\leq C_0e^{-\gamma T}$ for some $\gamma>0$, then
\begin{equation}\label{eq:conditionalrate}
 \dist_X(\calA_\nu,\calA_\mu)
 \leq C_\mu\,
 \varpi(\dBL(\mu,\nu))^{\gamma/(L+\gamma)}
\end{equation}
for $\dBL(\mu,\nu)$ sufficiently small.
\end{corollary}
\begin{proof}
For $0<q=\varpi(\dBL(\mu,\nu))<1$, choose
$T=(L+\gamma)^{-1}\log(q^{-1})$ in \eqref{eq:attractormodulus}. Both terms are then bounded by a constant times $q^{\gamma/(L+\gamma)}$. If $q=0$, the measures and attractors coincide.
\end{proof}

\begin{remark}\label{rem:attractorrate}
The additional attraction hypothesis in Corollary~\ref{cor:attractorrate} is not a consequence of the Halanay estimate. That estimate controls entry into a bounded region; convergence within that region to the global attractor can be slower. Existence of an exponential attractor, as in \cite{prazak2007}, also does not by itself give exponential attraction to the global attractor, which can be a proper subset. Similarly, the present arguments establish one-sided upper semicontinuity; lower semicontinuity requires additional dynamical hypotheses, as in the gradient-system results of \cite{haleraugel1989}.
\end{remark}

\begin{theorem}[Upper semicontinuity in $C(I;V)$]\label{thm:uscV}
Set $X_V=C(I;V)$. Every $\calA_\mu$ is a compact subset of $X_V$. For each $s\in(1/2,1)$ there is $C_s>0$ such that
\begin{equation}\label{eq:uscVestimate}
 \dist_{X_V}(\calA_\nu,\calA_\mu)
 \leq C_s\,
 \dist_X(\calA_\nu,\calA_\mu)^{\,1-\frac1{2s}},
 \qquad \mu,\nu\in\Kclass_{M_0}.
\end{equation}
In particular, weak-star convergence in $\Kclass_{M_0}$ implies
$\dist_{X_V}(\calA_{\mu_n},\calA_\mu)\to0$.
For each fixed $\mu$, the attractor also attracts every bounded $D\subset X$ in the $X_V$-semidistance, once the histories have entered $X_V$.
\end{theorem}
\begin{proof}
The spectral interpolation inequality is
\begin{equation}\label{eq:spectralinterpolation}
 \norm{A^{1/2}v}_H
 \leq\norm{v}_H^{\,1-\frac1{2s}}
       \norm{A^sv}_H^{\,\frac1{2s}},\qquad v\in D(A^s).
\end{equation}
For $\varphi\in\calA_\nu$ and $\psi\in\calA_\mu$, the common shift $z$ cancels in their difference, and \eqref{eq:attractorfractional} gives
$\norm{A^s(\varphi(\theta)-\psi(\theta))}_H\leq2R_s$.
Thus
\begin{equation}\label{eq:pairinterpolation}
 \norm{\varphi-\psi}_{X_V}
 \leq(2R_s)^{1/(2s)}
       \norm{\varphi-\psi}_X^{\,1-\frac1{2s}}.
\end{equation}
Each history belongs to $X_V$: its values lie in $V$ because
$z\in V$ and $\varphi(\theta)-z\in D(A^s)$, and \eqref{eq:spectralinterpolation}, applied to two times, upgrades its $H$-continuity to $V$-continuity.
Compactness in $X$ and \eqref{eq:pairinterpolation} then imply compactness in $X_V$. Taking the infimum over $\psi$ and the supremum over $\varphi$ proves \eqref{eq:uscVestimate}.

For the last assertion, Proposition~\ref{prop:smoothing} gives the same uniform shifted fractional bound on all sufficiently late histories from $D$. Interpolation therefore upgrades their attraction in $X$ to attraction in $X_V$.
\end{proof}

\begin{corollary}[Attractors under concentration and quadrature]
Suppose $\mu_n,\mu\in\Kclass_{M_0}$ and $\mu_n\weakstar\mu$, with
$\mu_n$ either absolutely continuous approximations or atomic quadratures. Their global attractors converge upper semicontinuously to $\calA_\mu$ in $X_V$, and hence also in $X$. In particular this applies to the concentration limits and quadratures in Corollaries~\ref{cor:distributedtodiscrete} and \ref{cor:quadrature} whenever \eqref{eq:dissipativecondition} holds for a common TV bound.
\end{corollary}
\begin{proof}
Apply Theorems~\ref{thm:usc} and \ref{thm:uscV}.
\end{proof}

\section{Sharpness and structured signed feedback}\label{sec:structured}

\begin{proposition}[Sharpness for the full coefficient class]\label{prop:sharpness}
If $\lambda_1+a_f\leq M_0b_1$, the constants in Assumption~\ref{ass:main} do not guarantee a bounded absorbing set or a compact global attractor for every $\mu\in\Kclass_{M_0}$.
\end{proposition}
\begin{proof}
Take $f(u)=a_fu$, $B(u)=b_1u$, $h=0$, and $\mu=M_0\delta_0$. These maps satisfy the structural assumptions with $c_0=0$ and any prescribed $b_0\geq0$. On the first eigenmode $u(t)=q(t)e_1$, the equation becomes
\[
 q'(t)=(M_0b_1-\lambda_1-a_f)q(t).
\]
If the coefficient is positive, there are unbounded forward trajectories. If it is zero, every constant history $\theta\mapsto qe_1$, $q\in\R$, is an equilibrium. A global attractor would have to contain this unbounded family, since it must attract every singleton equilibrium. Neither case admits the asserted uniform dissipative conclusion.
\end{proof}

The sharpness statement concerns a guarantee based only on the displayed constants. Particular feedback structures can yield better criteria. Consider the signed Pyragas feedback used, for example, in \cite{casasmateostroltzsch2018}:
\begin{equation}\label{eq:pyragas}
 \partial_tu+Au+f(u)
 =\kappa\left(\int_Iu(t+\theta)\,d\nu(\theta)-u(t)\right)+h,
 \qquad \nu\in\mathcal P(I),
\end{equation}
where $\mathcal P(I)$ is the set of probability measures and $\kappa\geq0$ is fixed.

\begin{theorem}[Robustness for arbitrary fixed negative-feedback gain]\label{thm:pyragas}
Assume that $h$ and $f$ satisfy the forcing and reaction hypotheses of Assumption~\ref{ass:main}, and that $\lambda_1+a_f>0$. For every fixed $\kappa\geq0$, the family \eqref{eq:pyragas}, parameterized by $\nu\in\mathcal P(I)$, has common compact absorption in $X$. Its global attractors are upper semicontinuous in $X_V$ under weak-star convergence of $\nu$. Finite-time solution histories satisfy the quantitative estimate \eqref{eq:BLrate} with a constant depending on $\kappa$.
\end{theorem}
\begin{proof}
Move $\kappa u(t)$ to the left and write
\[
 \partial_tu+Au+\widetilde f(u)
 =\int_I\widetilde B(u(t+\theta))\,d\widetilde\mu(\theta)+h,
\]
where
\[
 \widetilde f=f+\kappa I_H,\qquad
 \widetilde B=I_H,\qquad
 \widetilde\mu=\kappa\nu.
\]
The new reaction coefficient is $a_f+\kappa$, the feedback slope is one, and
$\norm{\widetilde\mu}_{\mathrm{TV}}=\kappa$. Consequently the damping condition is
\[
 \lambda_1+a_f+\kappa>\kappa,
\]
which follows from $\lambda_1+a_f>0$. All kernels $\kappa\nu$ belong to the same class with $M_0=\kappa$. Theorems~\ref{thm:attractorexistence}, \ref{thm:usc}, and \ref{thm:uscV} apply to this reformulation. Weak-star convergence of $\nu$ gives weak-star convergence of $\kappa\nu$. For the quantitative statement, use Theorem~\ref{thm:BLcontinuous} and
$\dBL(\kappa\nu,\kappa\widetilde\nu)=\kappa\dBL(\nu,\widetilde\nu)$,
together with $\varpi(\kappa d)\leq C_\kappa\varpi(d)$ when $\kappa>0$. The case $\kappa=0$ is independent of $\nu$.
\end{proof}

\begin{remark}
In the original representation, the signed kernel is $\kappa(\nu-\delta_0)$ and can have total variation $2\kappa$. Keeping its instantaneous negative part in the reaction exposes the extra damping. The constants above may depend on the fixed gain; no uniformity as $\kappa\to\infty$ is asserted. The theorem includes $\nu_n\weakstar\delta_0$, for which the entire feedback term vanishes in the limit.
\end{remark}

\section{Pointwise cubic reactions}\label{sec:cubic}

We now replace the abstract reaction hypothesis by a polynomial Nemytskii reaction. Throughout this section, $\Omega\subset\R^d$ is bounded and Lipschitz with $1\leq d\leq3$, $h\in V'$ is time-independent, and $B:H\to H$ satisfies the bounded-ball Lipschitz and affine-growth hypotheses of Assumption~\ref{ass:main}. Set
\begin{equation}\label{eq:cubicpolynomial}
 p(s)=a_3s^3+a_2s^2+a_1s+a_0,\qquad a_3>0,
 \qquad f(u)(x)=p(u(x)).
\end{equation}
The coefficients are fixed real numbers. We write $S_\mu^p$ for the resulting history semiflow. The delay lemmas use only the assumptions on $B$ and remain valid. No condition of the form \eqref{eq:dissipativecondition} is imposed in this section.

\subsection{Variational well-posedness and total-variation stability}

Write $Q_T=(0,T)\times\Omega$. The embeddings $V\hookrightarrow L^4(\Omega)$ and $V\hookrightarrow L^6(\Omega)$ will be used explicitly. The natural derivative space is
\[
 \mathcal Z_T=L^2(0,T;V')+L^{4/3}(Q_T)
 \ \hookrightarrow\ L^{4/3}(0,T;V'),
\]
where the sum is understood in the space of distributions.

\begin{definition}[Cubic weak solution]\label{def:cubicweak}
A solution of
\begin{equation}\label{eq:cubicequation}
 \partial_tu+Au+p(u)=\calD_\mu u+h,\qquad u|_I=\varphi,
\end{equation}
on $[0,T]$ satisfies
\[
 u\in C([-r,T];H)\cap L^2(0,T;V)\cap L^4(Q_T),
 \qquad \partial_tu\in\mathcal Z_T,
\]
and the equation holds in $V'$ for almost every time. In particular, the reaction term in its weak formulation is
$\int_\Omega p(u)v\,dx$ for $v\in V$. Initial histories need only belong to $X$; no $L^4$ condition is imposed on their negative-time values.
\end{definition}

\begin{lemma}[Cubic coercivity and difference inequalities]\label{lem:cubicalgebra}
There are positive constants $c_p,C_p,c_*,C_*$, depending only on $p$, such that
\begin{equation}\label{eq:cubiccoercivity}
 p(s)s\geq c_p|s|^4-C_p,\qquad
 p(s)s\geq k s^2-C_k\quad(s\in\R)
\end{equation}
for every prescribed $k\geq0$, with a corresponding finite $C_k\geq0$.
Let
\[
 \ell=\left(\frac{a_2^2}{3a_3}-a_1\right)_+.
\]
Then $p'(s)\geq-\ell$ and, for $a,b,\xi\in\R$,
\begin{align}
 (p(a)-p(b))(a-b)
 &\geq\frac{a_3}{4}|a-b|^4-\ell|a-b|^2,
 \label{eq:cubicmonotone}\\
 (p(a)-p(b))(a-b-\xi)
 &\geq c_*|a-b|^4-C_*|a-b|^2
       -C_*(1+a^2+b^2)|\xi|^2.
 \label{eq:cubicshiftineq}
\end{align}
\end{lemma}
\begin{proof}
The positive quartic leading term of $p(s)s$ proves \eqref{eq:cubiccoercivity}. For $w=a-b$ and $m=(a+b)/2$, direct expansion gives
\[
 (p(a)-p(b))w
 =\frac{a_3}{4}w^4+
 \left[3a_3\left(m+\frac{a_2}{3a_3}\right)^2
       +a_1-\frac{a_2^2}{3a_3}\right]w^2.
\]
This proves \eqref{eq:cubicmonotone}; completing the square in $p'$ gives the derivative bound. Polynomial expansion and Young's inequality also give, with $\Theta=1+a^2+b^2$,
\[
 |p(a)-p(b)|\leq C\Theta|w|,\qquad
 (p(a)-p(b))w\geq c\Theta w^2-Cw^2
\]
for suitable $c>0$ and $C$. Subtract the term $(p(a)-p(b))\xi$ and use
$C\Theta|w\xi|\leq(c/2)\Theta w^2+C'\Theta\xi^2$.
Since $\Theta\geq w^2/2$, the result is \eqref{eq:cubicshiftineq}.
\end{proof}

For differences of trajectories, define
\begin{equation}\label{eq:cubicerror}
 \mathfrak e_T(w)
 =\norm{w}_{C([-r,T];H)}^2
  +\norm{w}_{L^2(0,T;V)}^2+\norm{w}_{L^4(Q_T)}^4.
\end{equation}
This quantity records different powers of the natural norms and is not itself a norm.

\begin{theorem}[Global cubic well-posedness and stability]\label{thm:cubicwellposed}
Every $\mu\in\M(I)$ and $\varphi\in X$ determine a unique global cubic weak solution. For each $T>0$,
\begin{equation}\label{eq:cubicapriori}
 \sup_{t\in[-r,T]}\norm{u(t)}_H^2+
 \int_0^T\norm{u(t)}_V^2\,dt+\norm{u}_{L^4(Q_T)}^4
 \leq C_{T,D,M}
\end{equation}
uniformly for $\varphi$ in a bounded set $D\subset X$ and $\norm{\mu}_{\mathrm{TV}}\leq M$.

For $u=u_\mu^\varphi$, $v=u_\nu^\psi$, define the reference residual
\[
 \rho(t)=\int_I B(v(t+\theta))\,d(\mu-\nu)(\theta).
\]
For $\varphi,\psi\in D$ and $\norm{\mu}_{\mathrm{TV}},\norm{\nu}_{\mathrm{TV}}\leq M$,
\begin{equation}\label{eq:cubicresidual}
 \mathfrak e_T(u-v)
 \leq C_{T,D,M}\left(
     \norm{\varphi-\psi}_X^2+\norm{\rho}_{L^2(0,T;H)}^2\right).
\end{equation}
Consequently,
\begin{equation}\label{eq:cubicTV}
 \mathfrak e_T(u_\mu^\varphi-u_\nu^\psi)
 \leq C_{T,D,M}\left(
    \norm{\varphi-\psi}_X^2+\norm{\mu-\nu}_{\mathrm{TV}}^2\right).
\end{equation}
In particular, $S_\mu^p$ is a continuous semiflow, locally Lipschitz in the initial history, and total-variation Lipschitz stability holds in
$C([-r,T];H)\cap L^2(0,T;V)$.
\end{theorem}
\begin{proof}
We first construct solutions with $\varphi(0)\in V$ by a monotone-operator argument; this also justifies the estimates before passage to rough histories. Let $z=A^{-1}h\in V$ and write $w=u-z$. Put
\[
 \Psi(s)=\int_0^s p(\sigma)\,d\sigma+\frac{\ell}{2}s^2+C_\Psi.
\]
Choose $C_\Psi$ large enough that $\Psi(s)\geq c_\Psi|s|^4$ for some $c_\Psi>0$. This is possible because its leading coefficient is positive. Also $\Psi''=p'+\ell\geq0$. The functional
\begin{equation}\label{eq:cubicfunctional}
 \Phi_z(w)=
 \begin{cases}
 \displaystyle\frac12\norm{w}_V^2+\int_\Omega\Psi(w+z)\,dx,
     &w\in V,\\
 +\infty,&w\in H\setminus V,
 \end{cases}
\end{equation}
is proper, convex, and lower semicontinuous on $H$, with dense domain. Here $z\in V\hookrightarrow L^4(\Omega)$ is essential. Its Hilbert-space subgradient is the $H$-realization of
$Aw+p(w+z)+\ell(w+z)$. Thus \eqref{eq:cubicequation} becomes
\begin{equation}\label{eq:cubicsubgradient}
 \partial_tw+\partial\Phi_z(w)\ni
 P_\mu(t),\qquad P_\mu(t)=\ell u(t)+\calD_\mu u(t).
\end{equation}

For a prescribed continuous trajectory on the right, forced subgradient evolution is well posed and satisfies the contraction estimate
\[
 \norm{w_1(t)-w_2(t)}_H
 \leq\norm{w_1(0)-w_2(0)}_H
       +\int_0^t\norm{P_1(s)-P_2(s)}_H\,ds;
\]
see \cite{brezis1973}. Since $P_\mu$ is locally Lipschitz as a history functional, this estimate gives a contraction on a sufficiently short interval, with the prescribed negative-time history fixed. For initial data in the domain of $\Phi_z$, the resulting solution has
$\partial_tw\in L^2(0,T;H)$ and bounded $\Phi_z(w)$ on each compact interval of existence. This suffices for the following energy computations. An atom of $\mu$ at zero is included in this same history fixed-point argument.

Testing \eqref{eq:cubicequation} by $u$, using the quartic bound in \eqref{eq:cubiccoercivity}, and estimating the forcing in $V',V$ gives
\begin{equation}\label{eq:cubicenergy}
 \frac{d}{dt}\norm{u(t)}_H^2+\norm{u(t)}_V^2
       +c\norm{u(t)}_{L^4}^4
 \leq C\left(1+\sup_{s\in[t-r,t]}\norm{u(s)}_H^2\right).
\end{equation}
Delay Gronwall first bounds the $H$-norm; integration then bounds both remaining terms. The right side of \eqref{eq:cubicsubgradient} is bounded in $H$ on finite intervals. The subgradient energy inequality
\begin{equation}\label{eq:cubicPhi}
 \frac{d}{dt}\Phi_z(w(t))+
          \frac12\norm{\partial_tw(t)}_H^2
 \leq\frac12\norm{P_\mu(t)}_H^2
\end{equation}
precludes loss of the regularity needed for continuation from data in $V$. Thus these solutions are global.

For two such solutions, test their difference equation by $q=u-v$ and use \eqref{eq:cubicmonotone}. With
$Y_q(t)=\sup_{s\in[-r,t]}\norm{q(s)}_H^2$, Young's inequality and the delay Lipschitz bound yield
\[
 \frac{d}{dt}\norm{q(t)}_H^2+\norm{q(t)}_V^2
       +c\norm{q(t)}_{L^4}^4
 \leq C_{T,D,M}Y_q(t)+C\norm{\rho(t)}_H^2.
\]
Integration and Gronwall prove \eqref{eq:cubicresidual}.

For general $\varphi\in X$, take $\varphi_n(\theta)=E(1/n)\varphi(\theta)$.
Then $\varphi_n\to\varphi$ in $X$, $\varphi_n(0)\in V$, and the corresponding solutions are uniformly bounded by \eqref{eq:cubicapriori}. Estimate \eqref{eq:cubicresidual} with the same $\mu$ shows that they are Cauchy in
$C([-r,T];H)$, $L^2(0,T;V)$, and $L^4(Q_T)$. Polynomial growth gives
\begin{equation}\label{eq:cubicreactioncontinuity}
 \norm{p(u)-p(v)}_{L^{4/3}(Q_T)}
 \leq C_T\left(1+\norm{u}_{L^4(Q_T)}^2
                   +\norm{v}_{L^4(Q_T)}^2\right)
       \norm{u-v}_{L^4(Q_T)}.
\end{equation}
Hence the reaction terms converge in $L^{4/3}(Q_T)$; the other terms converge in the corresponding variational spaces. The limit has the prescribed continuous $H$-trace and solves \eqref{eq:cubicequation}.

For solutions in Definition~\ref{def:cubicweak}, the energy chain rule is justified by time regularization in
$L^2(0,T;V)\cap L^4(Q_T)$: the derivative decomposes into
$L^2(0,T;V')$ and $L^{4/3}(Q_T)$ terms, paired with those two spaces, respectively. Consequently the same difference inequality proves uniqueness and \eqref{eq:cubicresidual} for all cubic weak solutions. The growth bound on $B$ then proves \eqref{eq:cubicTV}. Time translation, uniqueness, and the continuous trajectories give the jointly continuous semiflow.
\end{proof}

\subsection{Quantitative weak-star stability for rough histories}

The logarithmic estimate for the abstract reaction cannot be applied directly to $p$: an $H$-ball does not bound its Lipschitz constant. Instead, we use the heat-propagated kernel residual as a corrector in the cubic energy estimate.

\begin{theorem}[Quantitative cubic kernel stability]\label{thm:cubicBL}
Under the hypotheses of this section, for every $T>0$, bounded $D\subset X$, and $M\geq0$,
\begin{equation}\label{eq:cubicBL}
 \mathfrak e_T(u_\mu^\varphi-u_\nu^\psi)
 \leq C_{T,D,M}\left(
 \norm{\varphi-\psi}_X^2+\varpi(\dBL(\mu,\nu))\right)
\end{equation}
whenever $\varphi,\psi\in D$ and both total variations are at most $M$. In particular,
\begin{align}\label{eq:cubicBLnorm}
 &\sup_{0\leq t\leq T}
  \norm{S_\mu^p(t)\varphi-S_\nu^p(t)\psi}_X
  +\norm{u_\mu^\varphi-u_\nu^\psi}_{L^2(0,T;V)}
 \notag\\
 &\qquad\leq C_{T,D,M}\left(
 \norm{\varphi-\psi}_X+
       \sqrt{\varpi(\dBL(\mu,\nu))}\right).
\end{align}
No temporal or spatial regularity beyond boundedness in $X$ is required of $D$.
\end{theorem}
\begin{proof}
Let $u=u_\mu^\varphi$, $v=u_\nu^\psi$, $q=u-v$, and
$\delta=\varpi(\dBL(\mu,\nu))$. Define $\rho$ as in Theorem~\ref{thm:cubicwellposed} and
\[
 R(t)=\int_0^t E(t-s)\rho(s)\,ds\quad(t\geq0),\qquad
 R(t)=0\quad(-r\leq t\leq0).
\]
Then $\partial_tR+AR=\rho$, $R(0)=0$. The common trajectory bound and the growth of $B$ give
$\norm{\rho}_{L^\infty(0,T;H)}\leq C_{T,D,M}$. Lemma~\ref{lem:integratedkernel}, applied to $g=B(v)$, gives
\begin{equation}\label{eq:cubiccorrectorH}
 \sup_{0\leq t\leq T}\norm{R(t)}_H\leq C_{T,D,M}\delta.
\end{equation}

For the Dirichlet heat semigroup in $d\leq3$,
\begin{equation}\label{eq:ultracontractive}
 \norm{E(t)}_{\mathcal L(H,L^\infty)}
 \leq C t^{-d/4}e^{-\lambda_1t/2},\qquad t>0.
\end{equation}
Indeed, domination by the Euclidean heat kernel gives the factor $t^{-d/4}$, and splitting $E(t)=E(t/2)E(t/2)$ combines it with the $H$ spectral gap. The right side is integrable in time because $d<4$. Consequently
\[
 \sup_{0\leq t\leq T}\norm{R(t)}_{L^\infty(\Omega)}
 \leq C_\Omega\norm{\rho}_{L^\infty(0,T;H)}
 \leq C_{T,D,M}.
\]
Interpolation and \eqref{eq:cubiccorrectorH} imply
\begin{equation}\label{eq:cubiccorrectorLfour}
 \sup_{0\leq t\leq T}\norm{R(t)}_{L^4}^2
 \leq\sup_t\norm{R(t)}_{L^\infty}\norm{R(t)}_H
 \leq C_{T,D,M}\delta.
\end{equation}
The linear energy identity also gives
\begin{equation}\label{eq:cubiccorrectorV}
 \int_0^T\norm{R(t)}_V^2\,dt
 \leq\int_0^T(\rho(t),R(t))_H\,dt
 \leq C_{T,D,M}\delta.
\end{equation}

Set $\eta=q-R$. Its equation is
\[
 \partial_t\eta+A\eta+p(u)-p(v)
 =\calD_\mu u-\calD_\mu v.
\]
Test by $\eta$. Applying \eqref{eq:cubicshiftineq} pointwise with
$(a,b,\xi)=(u,v,R)$ bounds the reaction contribution from below by
\[
 c_*\norm{q}_{L^4}^4-C_*\norm{q}_H^2
 -C_*\int_\Omega(1+u^2+v^2)R^2\,dx.
\]
The last integral is at most
\[
 C_\Omega\left(1+\norm{u}_{L^4}^2+\norm{v}_{L^4}^2\right)
       \norm{R}_{L^4}^2.
\]
Its time integral is bounded by $C_{T,D,M}\delta$, using
\eqref{eq:cubicapriori}, \eqref{eq:cubiccorrectorLfour}, and
Cauchy--Schwarz in time. The delay term is bounded by
$C_{T,D,M}(\norm{\eta(t)}_H^2+
 \sup_{s\in[-r,t]}\norm{q(s)}_H^2)$.
Since $q=\eta+R$, Gronwall gives
\[
 \norm{\eta}_{C([-r,T];H)}^2+
 \norm{\eta}_{L^2(0,T;V)}^2+
 \norm{q}_{L^4(Q_T)}^4
 \leq C_{T,D,M}
       \left(\norm{\varphi-\psi}_X^2+\delta\right).
\]
Here \eqref{eq:cubiccorrectorH} contributes $\delta^2$; it is absorbed into
$C_M\delta$ because $\dBL(\mu,\nu)\leq2M$. Add
\eqref{eq:cubiccorrectorH} and \eqref{eq:cubiccorrectorV} to obtain
\eqref{eq:cubicBL}. The case $M=0$ has $R=0$ and follows directly from
\eqref{eq:cubicresidual}.
\end{proof}

\begin{corollary}[Cubic concentration, quadrature, and energy convergence]\label{cor:cubiclimits}
If $\mu_n\weakstar\mu$ and $\varphi_n\to\varphi$ in $X$, then
\[
 u_{\mu_n}^{\varphi_n}\longrightarrow u_\mu^\varphi
 \quad\text{in }C([-r,T];H)\cap L^2(0,T;V)\cap L^4(Q_T),
\]
and the time derivatives converge in $L^{4/3}(0,T;V')$.
For a fixed common history, weak-star convergence in
$C([-r,T];H)\cap L^2(0,T;V)$ is uniform over bounded subsets of $X$.

Under the first-moment concentration assumption of
Corollary~\ref{cor:distributedtodiscrete}, the error in those two spaces is
$O(\sqrt{\varepsilon\log(e+\varepsilon^{-1})})$.
For the signed atomic quadrature of Corollary~\ref{cor:quadrature}, it is
$O(\sqrt{\Delta\log(e+\Delta^{-1})})$, uniformly on bounded history sets and bounded total-variation classes.
\end{corollary}
\begin{proof}
Weak-star convergence implies bounded total variation and convergence in
$\dBL$. Apply Theorem~\ref{thm:cubicBL}. For derivatives, subtract the equations, use \eqref{eq:cubicreactioncontinuity} and
$L^{4/3}(\Omega)\hookrightarrow V'$, and split the delay difference into its trajectory part and the fixed-reference residual. The first tends to zero in $C([0,T];H)$; the second does so by Lemma~\ref{lem:weakstaruniform}. The concentration and quadrature bounds follow from their corresponding $\dBL$ estimates.
\end{proof}

\subsection{Attractors without a smallness condition}

\begin{proposition}[Uniform cubic absorption and smoothing]\label{prop:cubicabsorbing}
For every fixed $M_0<\infty$, the semiflows $S_\mu^p$, $\mu\in\Kclass_{M_0}$, have a common positively invariant absorbing ball in $X$ and a common compact absorbing set $\calK_*^p\subset X$. For $z=A^{-1}h$ and every $0<s<1$, there is $R_s^p$ such that, for every bounded $D\subset X$,
\begin{equation}\label{eq:cubicfractional}
 \sup_{\mu\in\Kclass_{M_0}}\sup_{\varphi\in D}
 \norm{A^s(u_\mu^\varphi(t)-z)}_H\leq R_s^p
\end{equation}
for all sufficiently large $t$, with a common entry time for that $D$.
The radii and entry times may depend on $M_0$.
\end{proposition}
\begin{proof}
Choose $k\geq0$ with $\lambda_1+k>M_0b_1$. By
\eqref{eq:cubiccoercivity},
\[
 \int_\Omega p(u)u\,dx
 \geq k\norm{u}_H^2-C_k|\Omega|.
\]
The energy proof of Proposition~\ref{prop:Habsorbing} now applies with
$a_f=k$ and $c_0=C_k|\Omega|$: that proof uses only the energy identity and this lower bound, not an $H$-valued Lipschitz property of the reaction. It gives a common invariant absorbing history ball of radius $R_H^p$. Write $T_H(D)$ for its entry time.

On every unit interval contained in $[T_H(D),\infty)$, integration of
\eqref{eq:cubicenergy} gives a common bound for
\[
 \int_t^{t+1}\left(\norm{u(s)}_V^2+
                         \norm{u(s)}_{L^4}^4\right)\,ds.
\]
The shifted convex energy \eqref{eq:cubicfunctional} therefore satisfies
\begin{equation}\label{eq:cubicaveragePhi}
 \int_t^{t+1}\Phi_z(u(s)-z)\,ds\leq C,
 \qquad t\geq T_H(D).
\end{equation}
Indeed, $\Psi(s)\leq C(1+|s|^4)$ and
$\norm{u-z}_V^2\leq2\norm{u}_V^2+2\norm{z}_V^2$.
Moreover, $P_\mu=\ell u+\calD_\mu u$ is bounded in $H$ by a common constant $R_P$ on this time region.

Inequality \eqref{eq:cubicPhi} is valid on every positive-time interval starting at a time when $\Phi_z(u-z)$ is finite. Such times have full measure by \eqref{eq:cubicaveragePhi}; restarting the forced subgradient evolution there justifies the inequality for the rough initial histories as well. For $t\geq T_H(D)+1$ and almost every $s\in[t-1,t]$, it gives
\[
 \Phi_z(u(t)-z)\leq\Phi_z(u(s)-z)+\frac12R_P^2(t-s).
\]
Integrating in $s$ and using \eqref{eq:cubicaveragePhi} proves
\[
 \Phi_z(u(t)-z)\leq C+\frac14R_P^2.
\]
The bound extends to every such $t$ by lower semicontinuity. It yields a uniform eventual $V$-bound on $u$, because $\Phi_z(w)\geq\frac12\norm{w}_V^2$.

Now $V\hookrightarrow L^6(\Omega)$ gives
\begin{equation}\label{eq:cubicHforcing}
 \norm{p(u(t))}_H\leq C_p(1+\norm{u(t)}_V^3)\leq C.
\end{equation}
Thus $w=u-z$ solves
\[
 \partial_tw+Aw=-p(u)+\calD_\mu u
\]
with uniformly bounded $H$-valued forcing after entry. The unit-window heat-semigroup argument of Proposition~\ref{prop:smoothing} proves
\eqref{eq:cubicfractional} for $t\geq T_H(D)+2$. It also bounds $u$ in $V$ and $\partial_tu$ in $V'$ uniformly after this time. The interpolation estimate
\eqref{eq:timeholder} gives a common $1/2$-H\"older modulus in $H$ for late histories. After increasing the entry time by $r$, the compact-set construction of Proposition~\ref{prop:compactabsorbing} supplies $\calK_*^p$.
\end{proof}

\begin{theorem}[Cubic global attractors and kernel robustness]\label{thm:cubicattractors}
For every fixed $M_0<\infty$ and every $\mu\in\Kclass_{M_0}$, $S_\mu^p$ has a compact global attractor $\calA_\mu^p\subset X$. All these attractors lie in the common compact set $\calK_*^p$. They are also compact subsets of $X_V=C(I;V)$ and attract every bounded subset of $X$ in the $X_V$ semidistance after the histories have entered $X_V$.

If $\mu_n\weakstar\mu$ in $\Kclass_{M_0}$, then
\begin{equation}\label{eq:cubicUSC}
 \dist_X(\calA_{\mu_n}^p,\calA_\mu^p)\longrightarrow0,\qquad
 \dist_{X_V}(\calA_{\mu_n}^p,\calA_\mu^p)\longrightarrow0.
\end{equation}
These conclusions include distributed-to-discrete limits and signed atomic quadrature. No smallness restriction on $M_0b_1$, and no stronger forcing hypothesis than $h\in V'$, is required.
\end{theorem}
\begin{proof}
The continuous semiflow and common compact absorption give the attractors. Theorem~\ref{thm:cubicBL} gives uniform finite-time kernel convergence on $\calK_*^p$, so the invariance and attraction argument in Theorem~\ref{thm:usc} proves the first limit in \eqref{eq:cubicUSC}.

Invariance and \eqref{eq:cubicfractional} imply a common bound
\[
 \sup_{\mu\in\Kclass_{M_0}}\sup_{\varphi\in\calA_\mu^p}
 \sup_{\theta\in I}\norm{A^s(\varphi(\theta)-z)}_H\leq R_s^p,
 \qquad 0<s<1.
\]
For $1/2<s<1$, the common shift cancels in differences. Spectral interpolation therefore gives
\begin{equation}\label{eq:cubicUSCinterpolation}
 \dist_{X_V}(\calA_\nu^p,\calA_\mu^p)
 \leq (2R_s^p)^{1/(2s)}
 \dist_X(\calA_\nu^p,\calA_\mu^p)^{1-1/(2s)}.
\end{equation}
The same estimate for individual histories proves compactness in $X_V$ from compactness in $X$, and proves attraction in $X_V$ by comparing late trajectories with the attractor. This establishes the remaining claims.
\end{proof}

\begin{remark}[Quantitative attractor comparison]\label{rem:cubicattractorrate}
Put
\[
 a_\mu^p(T)=\dist_X(S_\mu^p(T)\calK_*^p,\calA_\mu^p).
\]
There are constants $C,L>0$, uniform on the fixed kernel class, such that
\begin{equation}\label{eq:cubicattractorrate}
 \dist_X(\calA_\nu^p,\calA_\mu^p)
 \leq\inf_{T\geq0}\left\{
 Ce^{LT}\sqrt{\varpi(\dBL(\mu,\nu))}+a_\mu^p(T)\right\}.
\end{equation}
To see the time dependence, start the finite-time comparison in $\calK_*^p$, contained in the common invariant $H$-ball. The $H$-bounds, delay Lipschitz constants, and corrector bounds are then independent of $T$, while
$\int_0^T\norm{u(t)}_{L^4}^2\,dt\leq C(1+T)$ follows from
\eqref{eq:cubicenergy}. The proof of Theorem~\ref{thm:cubicBL} thus has at most exponential growth in its constant. Invariance gives \eqref{eq:cubicattractorrate} as in Theorem~\ref{thm:usc}.
If the additional hypothesis $a_\mu^p(T)\leq C_0e^{-\gamma T}$ holds, optimization gives the conditional rate
$O(\varpi(\dBL(\mu,\nu))^{\gamma/[2(L+\gamma)]})$.
Exponential attraction to the global attractor is not assumed or deduced otherwise.
\end{remark}

\begin{example}[Allen--Cahn reaction]\label{ex:cubicAC}
For $p(s)=\gamma s^3-\beta s$, $\gamma>0$ and $\beta\in\R$,
\[
 p(s)s\geq k s^2-\frac{(k+\beta)_+^2}{4\gamma}
 \qquad\text{for every }k\geq0.
\]
Consequently the delayed equation
\[
 \partial_tu-\Delta u+\gamma u^3-\beta u
 =\int_I B(u(t+\theta))\,d\mu(\theta)+h
\]
has all the well-posedness, kernel-stability, and attractor properties above in dimensions one, two, and three. The feedback $B$ can be a globally Lipschitz pointwise map, the identity, or an admissible nonlinear nonlocal operator. The restriction $a_3>0$ expresses dissipative cubic growth; reactions with negative leading coefficient are not included.
\end{example}

\begin{remark}[Role of the dimension restriction]\label{rem:cubicdimension}
The cubic extension uses $d\leq3$ both to make the heat estimate
\eqref{eq:ultracontractive} integrable in time and to obtain the eventual
$H$-bound \eqref{eq:cubicHforcing} from a $V$-bound.
It does not identify $p$ with a locally Lipschitz map on $H$.
In particular, the derivative regularity at the initial time is
$\mathcal Z_T$, rather than the $L^2(0,T;V')$ regularity asserted for the abstract reaction class. Higher-dimensional polynomial problems require additional arguments and are not covered by the cubic theorems stated here.
\end{remark}

\section{Conclusion}

The delay law can be perturbed quantitatively in a topology that allows concentration and movement of atoms. For abstract $H$-valued reactions, the logarithmic stability estimate is uniform over bounded sets of continuous $L^2$-histories and yields explicit errors for distributed-to-discrete approximation and signed atomic quadrature. Total-variation stability and qualitative weak-star convergence also hold in the corresponding weak-solution energy norm.

For the abstract class, a gap between effective damping and the feedback growth slope supplies common compact absorption. The argument allows time-independent forcing in $H^{-1}(\Omega)$, and fractional smoothing strengthens attractor upper semicontinuity to continuous $H_0^1$-valued histories. The damping threshold is sharp as a guarantee based on those structural constants, while structured instantaneous negative feedback allows every fixed gain.

Pointwise cubic reactions in dimensions at most three admit a separate variational treatment. Combining cubic monotonicity with a heat-propagated kernel corrector gives the modulus $\sqrt{d\log(e+d^{-1})}$ for rough histories, together with convergence in the natural cubic energy spaces. Quartic coercivity removes the feedback smallness condition on every fixed total-variation-bounded class. A shifted convex energy and subsequent parabolic smoothing retain both attractor robustness results with forcing in $H^{-1}(\Omega)$.

Further questions include optimality of the two kernel moduli, higher-dimensional polynomial reactions, and construction of exponential attractors that vary quantitatively with the kernel. Infinite fading memory requires a different history space and is not covered by the present compact-interval argument.

\section*{Funding}
This research received no external funding.

\bibliographystyle{plainnat}
\bibliography{v3_refs}
\end{document}